\documentclass[letterpaper, 10 pt, conference]{ieeeconf}  

\IEEEoverridecommandlockouts                              

\usepackage{amsmath} 
\usepackage{amssymb}  
\usepackage{algorithm}
\usepackage{algpseudocode}

\usepackage{setspace}
\usepackage{makecell}  
\usepackage{multirow}
\usepackage{mathtools}

\usepackage{amsthm} 
\newtheorem{theorem}{Theorem}
\newtheorem{proposition}{Proposition}

\newtheorem{example}{Example}
\usepackage{nicefrac}   
\usepackage{mathtools}
\DeclarePairedDelimiter\ceil{\lceil}{\rceil}
\let\labelindent\relax
\usepackage[inline, shortlabels]{enumitem}
\usepackage[multiple]{footmisc}     
\usepackage{nth}

\usepackage{stfloats}
\newcommand\fnsep{\textsuperscript{,}}
\newcommand{\bmatrixAdjustSize}[2][.8]{%
  \scalebox{#1}{%
    \renewcommand{\arraystretch}{.9}%
    $\begin{bmatrix}#2\end{bmatrix}$%
  }
}

\usepackage[dvipsnames,table]{xcolor}
\colorlet{linkequation}{MidnightBlue}
\usepackage{etoolbox}
\makeatletter
\patchcmd{\@makecaption}
  {\scshape}
  {}
  {}
  {}
\makeatother
\usepackage{footnote}
\newcommand{\tp}{^{\top}}
\newcommand{\itp}{^{-\top}}
\newcommand{\iv}{^{-1}}
\newcommand\tpp[2][-6]{{#2}^{\mkern#1mu\top}} 

\makeatletter
\let\NAT@parse\undefined
\makeatother
\usepackage[colorlinks, citecolor=OliveGreen, linkcolor=Black]{hyperref}  
\usepackage{tablefootnote}
\usepackage[capitalize]{cleveref}
\crefname{equation}{}{}
\crefname{enumi}{}{}
\makeatletter
\creflabelformat{equation}{%
  \textup{%
    \hypersetup{
      linkcolor=linkequation,
      linkbordercolor=linkequation,
    }%
    (#2#1#3)%
  }%
}
\makeatother
\usepackage{float}
\usepackage{tikz}
\usetikzlibrary{fit}
\tikzset{%
  highlight/.style={rectangle,rounded corners,fill=red!15,draw,
    fill opacity=0.5,thick,inner sep=0pt}
}
\newcommand{\tikzmark}[2]{\tikz[overlay,remember picture,
  baseline=(#1.base)] \node (#1) {#2};}
\newcommand{\Highlight}[1][submatrix]{%
    \tikz[overlay,remember picture]{
    \node[highlight,fit=(left.north west) (right.south east)] (#1) {};}
}
\newcommand{\Highlightme}[1][]{%
    \tikz[overlay,remember picture]{
    \node[highlight,fit=(me.north west) (me.south east)] (#1) {};}
}
\title{\LARGE \bf
Brunovsky Riccati Recursion for Linear Model Predictive Control
}

\usetikzlibrary{shapes, arrows.meta, positioning}
\usetikzlibrary{decorations.markings}
\usepackage{amssymb}
\usepackage[format=hang,justification=raggedright]{caption}
\author{Shaohui Yang$^{1}$, Toshiyuki Ohtsuka$^{2}$, and Colin N. Jones$^{1}$
\thanks{This project has received funding from the European Union’s 2020 research and innovation programme under the Marie Skłodowska-Curie grant agreement No. 953348 ELO-X. Corresponding author: Shaohui Yang. }
\thanks{$^{1}$Automatic Control Laboratory, EPFL, Switzerland. 
{\tt\small shaohui.yang, colin.jones@epfl.ch} }%
\thanks{$^{2}$Department of Informatics, Graduate School of Informatics, Kyoto University, Kyoto, Japan.
        {\tt\small ohtsuka@i.kyoto-u.ac.jp}}%
}

\begin{document}

\maketitle
\thispagestyle{empty}
\pagestyle{empty}

\begin{abstract}

In almost all algorithms for Model Predictive Control (MPC), the most time-consuming step is to solve some form of Linear Quadratic (LQ) Optimal Control Problem (OCP) repeatedly. The commonly recognized best option for this is a Riccati recursion based solver, which has a time complexity of $\mathcal{O}(N(n_x^3 + n_x^2 n_u + n_x n_u^2 + n_u^3))$. In this paper, we propose a novel \textit{Brunovsky Riccati Recursion} algorithm to solve LQ OCPs for Linear Time Invariant (LTI) systems. The algorithm transforms the system into Brunovsky form, formulates a new LQ cost (and constraints, if any) in Brunovsky coordinates, performs the Riccati recursion there, and converts the solution back. Due to the sparsity (block-diagonality and zero-one pattern per block) of Brunovsky form and the data parallelism introduced in the cost, constraints, and solution transformations, the time complexity of the new method is greatly reduced to $\mathcal{O}(n_x^3 + N(n_x^2 n_u + n_x n_u^2 + n_u^3))$ if $N$ threads/cores are available for parallel computing. 

\end{abstract}

\allowdisplaybreaks

\section{INTRODUCTION}
The most common form of Model Predictive Control (MPC) is the Linear-Quadratic (LQ) Optimal Control Problem (OCP). A piece of LQ OCP for Linear Time Invariant (LTI) systems with $n_x$ states and $n_u$ inputs is given by
\begin{subequations}
    \begin{align}
        \underset{x, u}{\text{min}} & \quad \sum_{k=0}^{N-1} \frac{1}{2} 
        \begin{bmatrix}
         x_k \\
         u_k
        \end{bmatrix}\tp 
        \begin{bmatrix}
        Q_k & S_k\tp \\
        S_k & R_k
        \end{bmatrix}
        \begin{bmatrix}
         x_k \\
         u_k 
        \end{bmatrix} + 
        \begin{bmatrix}
         x_k \\
         u_k 
        \end{bmatrix}\tp 
        \begin{bmatrix}
        q_k \\
        r_k
        \end{bmatrix} \notag \\
        & \quad + 
        \frac{1}{2} x_N\tp Q_N  x_N +   x_N\tp q_N \label{eq:lq-ocp-cost} \\ 
        \text{s.t. } & \quad 
         x_{k+1} = A  x_k + B  u_k + b_k \quad x_0 \text{ given} \label{eq:lq-ocp-dynamics} \\
         & \quad C x_k + D u_k \leq d \label{eq:lq-ocp-inequality}
    \end{align}
    \label{eq:lq-ocp-all}
\end{subequations}
\noindent where $Q_k \in \mathbb{R}^{n_x \times n_x}, R_k \in \mathbb{R}^{n_u \times n_u}, S_k \in \mathbb{R}^{n_u \times n_x}$, and $q_k \in \mathbb{R}^{n_x}, r_k \in \mathbb{R}^{n_u}$ define the quadratic cost function. 
$C \in \mathbb{R}^{n_i \times n_x}, D \in \mathbb{R}^{n_i \times n_u}$, and $d \in \mathbb{R}^{n_i}$ define the $n_i$ linear inequalities.
The affine equality dynamical constraints are defined for the fixed linear system $A \in \mathbb{R}^{n_x \times n_x}, B \in \mathbb{R}^{n_x \times n_u}$ and stage-wise offsets $b_k \in \mathbb{R}^{n_x}$, and the problem has a prediction horizon of length $N$. 

The OCP~\cref{eq:lq-ocp-all} can be solved by posing it as a convex Quadratic Programming (QP) problem:
\begin{subequations}
    \begin{align}
        \underset{x}{\text{min}} & \quad \frac{1}{2} x\tp H x + x\tp g \\
        \text{s.t.} & \quad E_{e}x = f_{e} \quad E_{i}x \leq f_{i} \label{eq:ocp-qp-constraints}
    \end{align}
    \label{eq:ocp-qp}
\end{subequations}
Leaving the inequalities aside for a moment, there are at least three well-established ways~\cite{kouzoupis2018recent} to tackle the resulting KKT system defined as: 
\begin{equation}
    \begin{bmatrix}
        H & E_{e}\tp \\
        E_{e} & 0
    \end{bmatrix} \begin{bmatrix}
        x \\ 
        \lambda
    \end{bmatrix} = \begin{bmatrix}
        -g \\
        f_{e}
    \end{bmatrix}
    \label{eq:kkt-system}
\end{equation}

\textit{Condensing}~\cite{bock1984multiple}: Construct a null-space basis matrix $Z$ of $E_{e}$ before solving the linear system with the reduced Hessian $Z\tp H Z$. The name `condensing' originates from the elimination of states and the resulting reduction in decision variables. The computational cost of this approach is $\mathcal{O}(N^3 n_u^3)$ to factorize $Z\tp H Z$. 

\textit{Riccati recursion}, in which one uses a backward-forward recursion to solve the LQR-like problem entailing a cost of $\mathcal{O}(N (n_x^3+ n_u^3))$. This can be viewed as directly factoring the KKT matrix (which is block penta-diagonal if reordered, due to the inter-stage dynamical equality). This approach was first combined with the Interior Point Method (IPM) in~\cite{rao1998application} and a state-of-the-art Riccati based solver was reported in~\cite{frison2016algorithms} and implemented in HPIPM~\cite{frison2020hpipm}. 

\textit{Partial condensing}~\cite{axehill2015controlling} during which one applies the condensing method to consecutive blocks of size $M$. The strategy leads to a new OCP QP of $\ceil*{\frac{N}{M}}$ and will be block-wise dense. By adjusting $M$, one can control the level of sparsity and potential parallelism in the algorithm.

As reviewed above, progress has been made from the perspective of OCP QP structure. On the other hand, structures and transformations of the linear dynamics themselves have been well-studied but have seldom been exploited for solving such problems. For instance, Kalman decomposition~\cite{kalman1963mathematical} and state-feedback pole assignment~\cite{kautsky1985robust} are classical results. Nevertheless, their algorithmic developments (e.g. staircase algorithm~\cite{varga1981numerically,boley1981computing} and deadbeat gain computation~\cite{van1984deadbeat,sugimoto1993direct}) have received far less attention. In this paper, we exploit a lesser-known transformation: Brunovsky form~\cite{brunovsky1970classification}.

There are limited past works on taking advantage of time invariance when solving an OCP QP. The \textit{sparse condensing} approach injects sparsity to the reduced Hessian, favored by tailored linear solvers. Here we review three instances:

Banded null-space bases are assembled in~\cite{dang2017banded,yang2019system} to enforce bandedness of $Z\tp H Z$. The former one invokes the Turnback algorithm~\cite{berry1985algorithm} to a general system $(A, B)$ with no guarantee on bandwidth and effectiveness. Proofs are provided by the latter, which uses a Kalman decomposition of $(A, B)$ and constructs a two-sided deadbeat response using the controllable part of the system to fill in a custom $Z$. 

\cite{jerez2012sparse} formulates a banded reduced Hessian directly by finding a deadbeat gain $F$ to make $A+BF$ nilpotent with index $\mu$. Control $u_k$ thus has influence to at most $x_{k+\mu}$, assuming null-controllablity of the system. 

However, if inequality constraints $E_{i}, f_{i}$ are in place, all of the three will be only preferred by methods such as ADMM~\cite{o2013splitting} with fixed Hessian. In the more widely-used and effective interior point method, the Hessian in~\cref{eq:kkt-system} will be influenced by slacks and Lagrangian multipliers and becomes $H + \Delta H$. Re-computation of $Z\tp \Delta H Z$ is mandatory, even though $Z$ is fixed and the banded structure can be exploited. 

In this paper, we focus on the \textit{Riccati recursion} method because it incorporates with the interior point method well. Our contributions are two-fold: 
\begin{enumerate}[label=\textcolor{brown}{\arabic*.}]
     \item We propose to transform any $(A,B)$ first to controllable form and then to Brunovsky form and perform a Riccati recursion in these new coordinates. The transformations of costs, constraints, and solutions are done in parallel and the sparsity of Brunovsky form and the data parallelism together contribute to a Riccati solver with fastest big-O speed reported in the literature for LTI systems.
    \item We present a new perspective to transform any controllable system to Brunovsky form. 
\end{enumerate}

Notation: $I_{n} \in \mathbb{R}^{n \times n}$ denotes an identity matrix. $0_{m \times n}$ denotes a zero matrix. $\operatorname{blkdiag}$ is a short hand for a block diagonal matrix. $P \succ 0$ denotes positive definiteness of matrix $P$. $P \succeq 0$ denotes positive semi-definiteness. $[P]_{i, j}$ either represents a matrix block $P_{ij}$ or a scalar entry $p_{ij}$. 

\section{PRELIMINARIES}
\subsection{Riccati recursion for LQ OCP}
\vspace{-0.5em}
\begin{algorithm}[ht]
\begin{spacing}{1}
\caption{Riccati recursion to solve~\Crefrange{eq:lq-ocp-cost}{eq:lq-ocp-dynamics}.}
\begin{algorithmic}[1]
\Require $x_0, \{Q_k, S_k, R_k, q_k, r_k, b_k\}, A, B$
\State $P_N \gets Q_N$, $p_N \gets q_N$
\For{$k = N - 1 \rightarrow 0$}
\State $R_{e, k} \gets R_k + B\tp P_{k+1} B$
\State $K_k \gets -R_{e, k}\iv (S_k + B\tp P_{k+1} A)$
\State $P_k \gets Q_k + A\tp P_{k+1} A - K_k\tp R_{e, k} K_k$ \label{alg1:line:APA}
\State $k_k \gets  -R_{e, k}\iv (r_k + B\tp (P_{k+1} b_k + p_{k+1})) $
\State $p_k \gets q_k + A\tp (P_{k+1} b_k + p_{k+1}) - K_k\tp R_{e, k} k_k$
\EndFor
\For{$k = 0 \rightarrow N-1$}
\State $u_k^* \gets K_k x_k + k_k $ and  $x_{k+1}^* \gets A x_k + B  u_k + b_k$
\EndFor
\Ensure $\{x_k^*, u_k^*\}$
\end{algorithmic}
\label{alg1:riccati-classic}
\end{spacing}
\end{algorithm}
\vspace{-0.5em}
The classical Riccati recursion (\cref{alg1:riccati-classic}) solves~\Crefrange{eq:lq-ocp-cost}{eq:lq-ocp-dynamics} efficiently, with 
different variants and respective per-iteration floating point operations counts (up to cubic terms)
\begin{equation}
    \alpha n_x^3 + \beta n_x^2 n_u + \gamma n_x n_u^2 + \delta n_u^3
    \label{eq:flop-count-riccati}
\end{equation}
\noindent summarized in~\cref{table:riccati-variants}~\cite{frison2016algorithms}. 
\begin{table}[htbp]
\caption{~\cref{alg1:riccati-classic} with different assumptions and structure exploitation~\cite{frison2016algorithms}. }\label{table:riccati-variants}
\centering
\scalebox{0.95}{
\begin{tabular}{ |c|c|c| } 
 \hline
 Type & \makecell{Complexity~\cref{eq:flop-count-riccati} \\ $\delta=\nicefrac{1}{3}$} & Note \\ 
 \hline
 General & \makecell{$\alpha=4 $ \\ $\beta=6, \gamma=3$} & Cholesky factorize $R_{e, k}$\tablefootnote{Requires $R_k \succ 0, \forall k$, which, together with $\begin{bmatrix}
     Q_k & S_k\tp \\
     S_k & R_k
 \end{bmatrix} \succeq 0, \forall k$ and $Q_N \succeq 0$, are sufficient for the existence and uniqueness of~\Crefrange{eq:lq-ocp-cost}{eq:lq-ocp-dynamics} and thus are commonly assumed. } \\ 
 \hline 
 Symmetry & \makecell{$\alpha=3 $ \\ $\beta=5, \gamma=3$} & Exploit symmetry of $P_k$ \\ 
 \hline 
 Square-root & \makecell{$\alpha=\nicefrac{7}{3} $ \\ $\beta=4, \gamma=2$} & Cholesky factorize $P_k$\tablefootnote{Requires $P_k \succ 0, \forall k$. A sufficient condition for this is that $\begin{bmatrix}
     Q_k & S_k\tp \\
     S_k & R_k
 \end{bmatrix} \succ 0, \forall k$ and $P_N \succ 0$, which is quite common in practice. } \\ 
 \hline
\end{tabular}
}
\vspace{-1em}
\end{table}
\cref{alg1:riccati-classic} is applicable to time-varying systems such as $(A_k, B_k)$, but as will be shown in later sections, a fixed $(A, B)$ introduces a significant speed-up if proper transformations are applied before the recursion. 

\subsection{Transformations of linear systems}
\subsubsection[State transformations]{State transformations\footnote{Explanation on super/sub-scripts: $kd \rightarrow$ Kalman decomposition. $co \rightarrow$ controllable. $uc \rightarrow$ uncontrollable. $ca \rightarrow$ controllable canonical.}}
Given an arbitrary LTI system $(A, B)$, the Kalman decomposition~\cite{kalman1963mathematical} separates the controllable and uncontrolllable parts of the system with a state transformation $T_{kd} \in \mathbb{R}^{n_x \times n_x}$:
\begin{equation}
     \begin{bmatrix}
        A_{co} & A_{12} \\
        0 & A_{uc}
    \end{bmatrix} = T_{kd} A T_{kd}\iv \qquad 
    \begin{bmatrix}
        B_{co} \\ 0
    \end{bmatrix} = T_{kd} B
    \label{eq:kalman-decomposition-controllable}
\end{equation} 
where $A_{co} \in \mathbb{R}^{n_x^{c} \times n_x^{c}}, B_{co} \in \mathbb{R}^{n_x^{c} \times n_u}$ characterizes the $n_x^{c}$ controllable portion of the states and $A_{uc} \in \mathbb{R}^{n_x^{uc} \times n_x^{uc}}$ represents the self-evolving $n_x^{uc}$ uncontrollable part. 

The non-singular transformation $T_{kd}$ is not unique and can be developed in different ways. The staircase algorithm~\cite{varga1981numerically,boley1981computing} (\texttt{ctrbf} in Matlab) is one numerically reliable option that returns a well-conditioned unitary $T_{kd}$. The resulting $\begin{bmatrix}
    B_{co} & A_{co}
\end{bmatrix}$ is a staircase-like block Hessenberg matrix. 

Starting from the controllable pair $(A_{co}, B_{co})$, constructive methods~\cite{antsaklis1997linear, datta1977algorithm} are available to further transform the system to controllable canonical (companion) form with $T_{ca} \in \mathbb{R}^{n_x^{c} \times n_x^{c}}$. $A_{ca} = T_{ca} A_{co} T_{ca}\iv, B_{ca} = T_{ca} B_{co}$. The controllability indices $\{\mu_i\}_{i=1}^{n_u}, \sum_{i=1}^{n_u} \mu_i = n_x$ characterize the sparsity pattern of $A_{ca}, B_{ca}$ as follows: 
\begin{equation}
\begin{aligned}
    [A_{ca}]_{i, i} &= \begin{bmatrix}
        0_{(\mu_i-1) \times 1} & I_{\mu_i-1} \\
        \star & \star_{1 \times (\mu_i - 1)}
    \end{bmatrix} \, 
     [A_{ca}]_{i, j} = \begin{bmatrix}
        0_{(\mu_i-1) \times \mu_j}  \\
        \star_{1 \times \mu_j}
    \end{bmatrix}
    \\
    [B_{ca}]_{i, i} &= 
    \begin{bmatrix}
        0_{(\mu_i-1) \times 1} \\
        1
    \end{bmatrix} \quad 
    [B_{ca}]_{i, j} = \begin{bmatrix}
        0_{(\mu_i-1) \times 1} \\
        \star
    \end{bmatrix} \quad i < j
\end{aligned}
\label{eq:controllable-canonical-form}
\end{equation}
where $1 \leq i, j \leq n_u$, $[A_{ca}]_{i,j} \in \mathbb{R}^{\mu_i \times \mu_j}, [B_{ca}]_{i,j} \in \mathbb{R}^{\mu_i \times 1}$ and $\star$ denotes non-fixed entries with proper sizes~\cite{antsaklis1997linear}. $[B_{ca}]_{i,j} = 0_{\mu_i \times 1}, i > j$. Except the zeros and identity blocks in~\cref{eq:controllable-canonical-form} which are favored, $A_{ca}$ still has $n_x n_u$ dense entries out of $n_x^2$ and $B_{ca}$ has $\frac{n_u (n_u+1)}{2}$ out of $n_x n_u$. The method proposed in~\cite{datta1977algorithm} synthesizes $T_{ca}$ to achieve a minimum number of non-fixed parameters for $A_{ca}$, which still needs $n_x + n_u(n_u - 1)$ out of $n_x^2$ in the best case. As demonstrated in the next section, transformation to Brunovsky form solves the problem at its root. 

\subsubsection{Feedback transformations}

It is common knowledge that the eigenvalues of the closed-loop matrix $A_{co} + B_{co} F$ can be arbitrarily placed by a properly designed feedback gain $F \in \mathbb{R}^{n_u \times n_x^{c}}$ for a controllable pair $(A_{co}, B_{co})$~\cite{kautsky1985robust}. $F$ is often referred as a feedback transformation. 

An interesting instance of eigenvalue placement is to introduce nilpotency (placing all eigenvalues on the origin), i.e., $(A_{co} + B_{co} F_{db})^{\mu} = 0$ where the controllability index $\mu = \max_{i} \mu_i$ coincides with the index for nilpotency. Such a design is called \textit{deadbeat} control in the literature. Numerically stable algorithms that formulate a deadbeat feedback gain $F_{db}$ using staircase $(A_{co}, B_{co})$ are available~\cite{van1984deadbeat,sugimoto1993direct}. 

\section{LTI SYSTEM TO BRUNOVSKY FORM}
\label{sec:lti-to-brunovsky}

First proposed in~\cite{brunovsky1970classification}, the Brunovsky form is characterized by a set of controllability indices $\{\mu_i\}_{i=1}^{n_u}$: 
\begin{subequations}
\begin{align}
    A_{b} &= \operatorname{blkdiag}(A_i) \quad 
    A_{i} = \begin{bmatrix}
        0_{(\mu_i-1) \times 1} & I_{\mu_i-1} \\
        0 & 0_{1 \times (\mu_i-1)}
    \end{bmatrix} \label{eq:brunovsky-form-A}
    \\
    B_{b} &= \operatorname{blkdiag}(B_i) \quad 
    B_i = \begin{bmatrix}
        0_{(\mu_i-1) \times 1} \\
        1
    \end{bmatrix} \label{eq:brunovsky-form-B}
\end{align}
\label{eq:brunovsky-form}
\end{subequations}
The linear system $(A_b, B_b)$ in~\cref{eq:brunovsky-form} is composed of $n_u$ independent chain-of-integrator dynamics, each of size $\mu_i$, with an input to the $i$th derivative of corresponding state. It is of similar structure to the canonical form in~\cref{eq:controllable-canonical-form} but with zero off-diagonal blocks ($[A]_{i,j} = [B]_{i,j} = 0, \forall i \neq j$) and zero last rows of the diagonal blocks $A_i, B_i$ (all $\star$ are gone). 

\cite{brunovsky1970classification} proves that any controllable pair $(A_{co}, B_{co})$ can be transformed to~\cref{eq:brunovsky-form} with $T, F, G$ of proper sizes and $T, G$ being nonsingular (named as \textit{feedback equivalence})
\begin{equation}
    A_b = T (A_{co} + B_{co}F) T\iv \qquad B_b = T B_{co} G
    \label{eq:transformation-brunovsky}
\end{equation} 

Existing methods construct $T$ first and then $F, G$.
\begin{enumerate}[label=\textcolor{brown}{\arabic*.}]
    \item\label{enu:brunovsky-method-hessenburg}\cite{ford1984simple} proposed to transform $(A_{co}, B_{co})$ to a block triangular system first (with $T$) and then to Brunovsky form with $u = Fx + Gv$ where $v \in \mathbb{R}^{n_u}$ denotes the new input. However, the algorithm lacks a correctness proof of and contains many typos in pseudo-code.
    \item\label{enu:brunovsky-method-canonical}It is straightforward to obtain Brunovsky form from the controllable canonical form~\cref{eq:controllable-canonical-form} by eliminating the last rows of $[A_{ca}]_{i,j}, [B_{ca}]_{i,j}$ using $u = Fx + Gv$. Different $T_{ca}$ are available to transform $(A_{co}, B_{co})$ to $(A_{ca}, B_{ca})$.
\end{enumerate}

Here we provide the following new perspective to compute $F$ first, then $T$ to get $A_b$ in~\cref{eq:brunovsky-form-A}. With the super-diagonal entry of $A_i$ being $1$, $A_b$ can be viewed as the Jordan normal form of $A_{co} + B_{co}F$ (up to the order of the controllability indices $\mu_i$), with $T$ being the similarity transformation. Since $A_b$ is nilpotent, i.e., $A_b^{\mu} = 0$, so is $A_{co} + B_{co}F$. The feedback gain $F$ achieves deadbeat control. Such a $F_{db}$ can be obtained from staircase controllable $(A_{co}, B_{co})$~\cite{sugimoto1993direct}. Then $T_{jo}$ is devised to transform the closed-loop matrix to $A_b$\footnote{Explanation on subscript: $jo \rightarrow$ Jordan normal form. $db \rightarrow$ deadbeat gain. From now on, $T, F$ in~\cref{eq:transformation-brunovsky} will be annotated as $T_{jo}, F_{db}$ to highlight their intrinsicalities regardless of how they are constructed. }. However, as demonstrated by the following example, such strategy is insufficient to guarantee the existence of $G$ in~\cref{eq:transformation-brunovsky}. 
\begin{example}
    Consider a controllable pair $A_{co} = \begin{bsmallmatrix}
        0 & -1 \\
        1 & -1
    \end{bsmallmatrix}, B_{co} = \begin{bsmallmatrix}
        1 \\ 
        0
    \end{bsmallmatrix}$. It is easy to verify that the only deadbeat gain is $F_{db} = \begin{bsmallmatrix}
        1 & 0
    \end{bsmallmatrix}$ which makes $A_{co} + B_{co}F_{db} = $\raisebox{-0.1em}{$\begin{bsmallmatrix}
        1 & -1 \\
        1 & -1
    \end{bsmallmatrix}$}. All transformations that results in $T_{jo} (A_{co} + B_{co}F_{db}) T_{jo}\iv = \begin{bsmallmatrix}
        0 & 1 \\
        0 & 0
    \end{bsmallmatrix} = A_b$, which is the only Jordan normal nilpotent $2 \times 2$ matrix, are of the form $T_{jo} = \theta \begin{bsmallmatrix}
        \frac{1}{2} & -\frac{1}{2} \\
        1 & 1
    \end{bsmallmatrix}, \theta \neq 0 $. As a result, $T_{jo} B_{co} = \theta \begin{bsmallmatrix}
        \frac{1}{2} \\
        1
    \end{bsmallmatrix}$. Since $n_x=2, n_u=1$, the only possible $B_b$ is $\begin{bsmallmatrix}
        0 \\
        1
    \end{bsmallmatrix}$. Hence, it is impossible to find a $G \in \mathbb{R}$ s.t. $T_{jo} B_{co} G = B_b$. If $n_u = n_x$, invertibility of $T_{jo} B_{co}$ guarantees the existence of non-singular $G$.  
\end{example}

\section{BRUNOVSKY RICCATI RECURSION}
In this section, we present the main result of the paper: transforming an LTI system to Brunovsky form accelerates the solving of an LQ OCP. 
\subsection{General \texorpdfstring{$(A, B)$}{Lg} to controllable  \texorpdfstring{$(A_{co}, B_{co})$}{Lg}}
The motivation of this subsection is "stop wasting": not all LTI systems are fully controllable. If not, doing Riccati recursion with the uncontrollable part is unnecessary. 

Starting from an arbitrary pair $(A, B)$, we first Kalman decompose the system into~\cref{eq:kalman-decomposition-controllable} with $T_{kd}$. Denote the new state representation as $\begin{bmatrix}
    \tpp{x^c} & \tpp{x^{uc}}
\end{bmatrix}\tp = T_{kd} x$. The dynamics~\cref{eq:lq-ocp-dynamics} is equivalently transformed to the following: 
\begin{subequations}
\begin{align}
    x_{k+1}^{c} &= A_{co} x_k^c + A_{12} x_k^{uc} + B_{co} u_k^c + b_k^c \label{eq:controllable-dynamics} \\
    x_{k+1}^{uc} &= A_{uc} x_k^{uc} + b_k ^{uc} \label{eq:uncontrollable-dynamics}
\end{align}
\end{subequations}
where $x_k^c \in \mathbb{R}^{n_x^{c}}$ denotes the controllable part of the states affected by input $u_k = u_k^c$ and $x_k^{uc} \in \mathbb{R}^{n_x^{uc}}$ denotes the uncontrollable part that evolves on its own. $u_k^c$ is introduced only for clearer notation. The new offsets are defined as: 
\begin{equation}
    \begin{bmatrix}
    \tpp{b_k^c} & \tpp{b_k^{uc}}
    \end{bmatrix}\tp = T_{kd} b_k
    \label{eq:new-offset-controllable}
\end{equation}
Naturally, instead of solving an OCP of size $(n_x, n_u)$ as in~\Crefrange{eq:lq-ocp-cost}{eq:lq-ocp-dynamics}, it is advisable to solve an equivalent one but (potentially) of smaller size $(n_x^{c}, n_u)$:
\begin{subequations}
    \begin{align}
    \underset{x^c, u^c}{\text{min}} & \quad \sum_{k=0}^{N-1} \frac{1}{2} 
    \begin{bmatrix}
     x_k^c \\
     u_k^c
    \end{bmatrix}\tp 
    \begin{bmatrix}
    Q_k^c & \tpp{S_k^c} \\
    S_k^c & R_k
    \end{bmatrix}
    \begin{bmatrix}
     x_k^c \\
     u_k^c
    \end{bmatrix} + 
    \begin{bmatrix}
     x_k^c \\
     u_k^c
    \end{bmatrix}\tp 
    \begin{bmatrix}
    q_k^c \\
    r_k
    \end{bmatrix} \notag \\
    & \quad + 
    \frac{1}{2} \tpp{x_N^c} Q_N^c  x_N +   \tpp{x_N^c} q_N^c \\ 
    \text{s.t. } & \quad 
     x_{k+1}^c = A_{co} x_k^c + B_{co} u_k^c + b_k^{co} \quad x_0^c \text{ given}
\end{align}
\label{eq:lqr-controllable}
\end{subequations}
The new terms in~\cref{eq:lqr-controllable} are defined as follows: 
\begin{subequations}
    \begin{align}
        Q_k^c &\in \mathbb{R}^{n_x^{c} \times n_x^{c}} & \begin{bmatrix}
            Q_k^c & \star \\
            \star & \star
        \end{bmatrix} &= T_{kd}\itp Q_k  T_{kd}\iv \label{eq:new-controllable-cost-Q} \\
        S_k^c &\in \mathbb{R}^{n_u \times n_x^{c}} & \begin{bmatrix}
            \tpp{S_k^c} & \star
        \end{bmatrix}\tp &= S_k T_{kd}\iv \\
        q_k^c &\in \mathbb{R}^{n_x^{c}} & \begin{bmatrix}
            \tpp{q_k^c} & \star
        \end{bmatrix}\tp &= T_{kd}\itp q_k \\
        b_k^{co} &\in \mathbb{R}^{n_x^{c}} & b_k^{co} &= b_k^c + A_{12}x_k^{uc} \label{eq:new-controllable-cost-offset} \\
        x_0^c &\in \mathbb{R}^{n_x^{c}} & \begin{bmatrix}
            \tpp{x_0^c} & 
            \tpp{x_0^{uc}}
        \end{bmatrix}\tp &= T_{kd}x_0 \label{eq:new-controllable-cost-initial-state}
    \end{align}%
    \label{eq:new-controllable-cost}%
\end{subequations}%
\noindent where $\star$ denotes matrix or vector blocks of proper sizes but of no use. $R_k$ and $r_k$ remain unchanged compared with~\cref{eq:lq-ocp-cost}. 

\begin{proposition}
    Assume the uncontrollable trajectory $\{x_k^{uc}\}$ is available. The solution to~\Crefrange{eq:lq-ocp-cost}{eq:lq-ocp-dynamics} $\{x_k^{*}, u_k^*\}$ can be obtained from the solution to~\cref{eq:lqr-controllable} $\{x_k^{c*}, u_k^{c*}\}$ via: 
    \begin{equation}
        x_k^* = T_{kd}\iv \begin{bmatrix}
        \tpp{x_k^{c*}} &
        \tpp{x_k^{uc}}
    \end{bmatrix}\tp \qquad u_k^* = u_k^{c*}
    \label{eq:solution-recover-controllable}
    \end{equation}
\end{proposition}
\begin{proof}
    The result follows directly from the coordinate transformations~\cref{eq:kalman-decomposition-controllable,eq:solution-recover-controllable}, the isolation between controllable/uncontrollable trajectories~\Crefrange{eq:controllable-dynamics}{eq:uncontrollable-dynamics}, and the formulations for new offset and cost coefficients in~\cref{eq:new-offset-controllable,eq:new-controllable-cost}. 
\end{proof}

\subsection{Controllable \texorpdfstring{$(A_{co}, B_{co})$}{Lg} to Brunovsky \texorpdfstring{$(A_{b}, B_{b})$}{Lg}}
As pointed out in~\cite{frison2016algorithms}, the most computationally heavy part of~\cref{alg1:riccati-classic} is $A\tp P_{k+1} A$ in~\cref{alg1:line:APA} which costs $\alpha n_x^3$ FLOPs. In the following, the subscript of the cost-to-go matrix $P_{k+1}$ will be dropped for convenience. \cref{table:riccati-variants} states that general matrix-matrix multiplication (a so-called \textit{gemm} operation) has $\alpha = 4$. Symmetry of $P$ means $P = \Pi + \Pi\tp$ with triangular matrix $\Pi$, followed by a triangular-matrix multiplication (named \textit{trmm}) on $\Pi A$ and a \textit{gemm} on $A\tp (\Pi A)$, thus makes $\alpha = 3$. If all $P_k$ are assumed to be positive definite, the Cholesky facotization (\textit{potrf}) will return $P = L L\tp$. Followed by \textit{trmm} on $L\tp A$ and symmetric rank-k update (\textit{syrk}) on $(L\tp A)\tp (L\tp A)$, $\alpha = \nicefrac{7}{3}$ can be achieved, which is the current best practice\footnote{\textit{gemm} $\rightarrow$ general matrix-matrix multiplication. \textit{trmm} $\rightarrow$ triangular matrix-matrix multiplication. \textit{potrf} $\rightarrow$ Cholesky decomposition (positive definite matrix triangular factorization). \textit{syrk} $\rightarrow$ symmetric rank-k update. All abbreviations follow the standard BLAS~\cite{blackford2002updated} Level-3 API. }\fnsep\footnote{$\beta, \gamma$ in~\cref{eq:flop-count-riccati} are also reduced in different settings, but they play a secondary role due to $n_x \gg n_u$. See more detailed explanation in~\cite{frison2016algorithms}. }. 

The Brunovsky form~\cref{eq:brunovsky-form} avoids all aforementioned calculations, whose sparsity can be interpreted in two levels:
\begin{enumerate}[label=\textcolor{brown}{\arabic*.}]
    \item Block-diagonality of $(A_b, B_b)$.
    \item Zero-one patterns per block. $A_i$ of $A_b$ has an sole identity block on the top-right corner, so it acts like \textit{shifting} in \textit{gemm}. $B_i$ of $B_b$ is a unit vector with size $\mu_i$, so it acts like \textit{selection} of rows or columns in \textit{gemm}. 
\end{enumerate}
Due to block-diagonality, $A_b\tp P A_b$, $B_b\tp P B_b$, and $B_b\tp P A_b$ can be assembled block-wise in parallel. Due to zero-one patterns, each block is constructed by \textbf{selective memory copy-paste}, which is also parallelizable. \textbf{No floating point operations} are necessary. The dominating $\alpha n_x^3$ FLOPs in~\cref{eq:flop-count-riccati} reduces to $n_x^2$ memory copy in time. 
The coefficients $\beta, \gamma$ shrink with the same reason. \cref{thm:sparsity-copy} provides the closed-form copy-paste formula. 

\begin{theorem}
\label{thm:sparsity-copy}
    Denote $X, Y \in \{A, B\}$, i.e., $X, Y$ represent $A, B$ interchangeably. Denote $P_{ij} = [P]_{i, j} \in \mathbb{R}^{\mu_i \times \mu_j}$, i.e., the cost-to-go matrix $P$ is viewed as a $n_u \times n_u$ block matrix. The following formula returns $A_b\tp P A_b$, $B_b\tp P B_b$, and $B_b\tp P A_b$ making use of block-diagonality: 
    \begin{equation}
        X_i\tp P_{ij} Y_j = [X_b\tp P Y_b]_{i, j}
        \label{eq:sparsity-block}
    \end{equation} 
    As a result of the zero-one pattern, each block in~\cref{eq:sparsity-block} is copied from $P$ following the rules below: 
    \begin{subequations}
    \begin{align}
        [A_i\tp P_{ij} A_j]_{m, n} &= [P_{ij}]_{m-1, n-1} & m, n \geq 2 \\
        [B_i\tp P_{ij} A_j]_{1, n} &= [P_{ij}]_{\mu_i, n-1} & n \geq 2 \\
        B_i\tp P_{ij} B_j &= [P_{ij}]_{\mu_i, \mu_j} 
    \end{align}
    \label{eq:sparsity-pattern}
    \end{subequations} 
    Entries are zero when $m=1$ or $n=1$. 
\end{theorem}
\begin{proof}
    The result follows directly from plain matrix-matrix multiplication (\textit{gemm}) because of the block-diagonality and zero-one pattern per block of $(A_b, B_b)$ in~\cref{eq:brunovsky-form}. 
\end{proof}
\noindent The following example elaborates~\cref{thm:sparsity-copy} intuitively. 
\begin{example}
    Here we demonstrate the benefits of zero-one pattern. Let $\mu_i = 3, \mu_j = 4$, $p_{ij}$ denotes the entries of $P_{ij} \in \mathbb{R}^{3 \times 4}$. Three cases of~\cref{eq:sparsity-pattern} are conducted: 
    \begin{align} \label{eq:sparsity-example}
        A_i\tp P_{ij} A_j &= \bmatrixAdjustSize{
            0 & 0 & 0 \\
            1 & 0 & 0 \\
            0 & 1 & 0 \\
        } \begin{bmatrix}
            \tikzmark{left}{$p_{11}$} & p_{12} & p_{13} & p_{14} \\
            p_{21} & p_{22} & \tikzmark{right}{$p_{23}$} & p_{24} \\
            p_{31} & p_{32} & p_{33} & p_{34} \Highlight[first] \\
        \end{bmatrix} \bmatrixAdjustSize{
            0 & 1 & 0 & 0 \\
            0 & 0 & 1 & 0 \\
            0 & 0 & 0 & 1 \\
            0 & 0 & 0 & 0 \\
        } \notag \\
        &= \begin{bmatrix}
            0 & 0 & 0 & 0 \\
            0 & \tikzmark{left}{$p_{11}$} & p_{12} & p_{13} \\
            0 & p_{21} & p_{22} & \tikzmark{right}{$p_{23}$} \\
        \end{bmatrix} \Highlight[second] \\
        B_i\tp P_{ij} A_j &= \begin{bmatrix}
            0 \\ 0 \\ 1
        \end{bmatrix}\tp \begin{bmatrix}
            p_{11} & p_{12} & p_{13} & p_{14} \\
            p_{21} & p_{22} & p_{23} & p_{24} \\
            p_{31} & \tikzmark{left}{$p_{32}$} & p_{33} & \tikzmark{right}{$p_{34}$} \notag \\
        \end{bmatrix} \Highlight[third] \begin{bmatrix}
            0 & 1 & 0 & 0 \\
            0 & 0 & 1 & 0 \\
            0 & 0 & 0 & 1 \\
            0 & 0 & 0 & 0 \\
        \end{bmatrix} \notag \\
        &= \begin{bmatrix}
            0 & \,\, \tikzmark{left}{$p_{32}$} & \,\, p_{33} & \,\, \tikzmark{right}{$p_{34}$} & 
        \end{bmatrix}
        \Highlight[forth] \notag \\
        B_i\tp P_{ij} B_j &= \begin{bmatrix}
            0 \\ 0 \\ 1
        \end{bmatrix}\tp \begin{bmatrix}
            p_{11} & p_{12} & p_{13} & p_{14} \\
            p_{21} & p_{22} & p_{23} & p_{24} \\
            p_{31} & p_{32} & p_{33} & \tikzmark{me}{$p_{34}$} \\
        \end{bmatrix} \Highlightme[fifth] \begin{bmatrix}
            0 \\
            0 \\
            0 \\
            1
        \end{bmatrix} = \begin{bmatrix}
            \quad \tikzmark{me}{$p_{34}$} &
        \end{bmatrix} \Highlightme[sixth] \notag
    \end{align}
    \tikz[overlay,remember picture] {
      \draw[-latex,thick,red,line width=1pt] (first) -- (second) node [pos=0.5, right] {copy};
    }
    \tikz[overlay,remember picture] {
      \draw[-latex,thick,red,line width=1pt] (third) -- (forth) node [pos=-0.2, below] {copy};
    }
    \tikz[overlay,remember picture] {
      \draw[-latex,thick,red,line width=1pt] (fifth) -- (sixth) node [pos=0.8, below] {copy};
    }
    \noindent \cref{eq:sparsity-example} shows that $A_i\tp$ shifts $P$ downwards in $A_i\tp P$, $A_j$ shifts $P$ rightwards in $P A_j$, $B_i\tp$ selects last row of $P$ in $B_i\tp P$, and $B_j$ selects last column of $P$ in $P B_j$. 
\end{example}
As discussed previously, from the controllable $(A_{co}, B_{co})$, we calculate a similarity transformation $T_{jo}$, a deadbeat gain $F_{db}$, and an input transformation $G$ to devise the following Brunovsky LQ OCP that is similar to~\cref{eq:lqr-controllable}: 
\begin{subequations}
    \begin{align}
        \underset{z, v}{\text{min}} & \quad \sum_{k=0}^{N-1} \frac{1}{2} 
        \begin{bmatrix}
         z_k \\
         v_k
        \end{bmatrix}\tp 
        \begin{bmatrix}
        \Tilde{Q}_k & \Tilde{S}_k\tp \\
        \Tilde{S}_k & \Tilde{R}_k
        \end{bmatrix}
        \begin{bmatrix}
         z_k \\
         v_k 
        \end{bmatrix} + 
        \begin{bmatrix}
         z_k \\
         v_k 
        \end{bmatrix}\tp 
        \begin{bmatrix}
        \Tilde{q}_k \\
        \Tilde{r}_k
        \end{bmatrix} \notag \\
        & \quad + 
        \frac{1}{2} z_N\tp \Tilde{Q}_N z_N +   z_N\tp \Tilde{q}_N \\ 
        \text{s.t. } & \quad 
         z_{k+1} = A_{b} z_k + B_{b} v_k + \Tilde{b}_k \quad z_0 \text{ given}
    \end{align} 
    \label{eq:lqr-brunovsky}
\end{subequations}
where the state transformation $x^c = T_{jo}\iv z$ and transformations $u^c = F_{db} x^c + G v$ on input are applied simultaneously. The new terms in~\cref{eq:lqr-brunovsky} are defined as follows: 
\begin{subequations}
    \begin{align}
        \Tilde{Q}_k &= T_{jo}\itp (Q_k^c + F_{db}\tp R_k F_{db}) T_{jo}^{-1} \label{eq:new-brunovsky-cost-Q} \\
        \Tilde{R}_k &= G\tp R_k G \qquad \Tilde{r}_k = G\tp r_k \label{eq:new-brunovsky-cost-R-r} \\
        \Tilde{S}_k &= G\tp (S_k^c + R_k F_{db})T_{jo}^{-1} \label{eq:new-brunovsky-cost-S} \\
        \Tilde{q}_k &= T_{jo}\itp (q_k^c + F_{db}\tp r_k) \label{eq:new-brunovsky-cost-q} \\
        \Tilde{b}_k &= T_{jo} b_k^{co} \qquad z_0 = T_{jo} x_0^c \label{eq:new-brunovsky-cost-offset-initial-state}
    \end{align}
    \label{eq:new-brunovsky-cost}
\end{subequations}
\vspace{-1.5em}
\begin{proposition}
    The solution to~\cref{eq:lqr-controllable} $\{x_k^{c*}, u_k^{c*}\} $ can be obtained from the solution to~\cref{eq:lqr-brunovsky} $\{z_k^{*}, v_k^{*}\}$ via:
    \begin{equation}
        x_k^{c*} = T_{jo}\iv z_k^* \qquad 
        u_k^{c*} = F_{db}x_k^{c*} + G v_k^* 
    \label{eq:solution-recover-brunovsky}
    \end{equation}
\end{proposition}
\begin{proof}
    The result follows from the coordinate transformations~\cref{eq:transformation-brunovsky,eq:solution-recover-brunovsky}, and the new ingredients in~\cref{eq:new-brunovsky-cost}. 
\end{proof}
\subsection{Summary and remarks on the chain \texorpdfstring{\cref{eq:lq-ocp-all}~$\rightarrow$~\cref{eq:lqr-controllable}~$\rightarrow$~\cref{eq:lqr-brunovsky}}{}}
\setlength{\textfloatsep}{0pt}
Here we present the main contribution of this paper, the \textit{Brunovsky Riccati recursion}, in~\cref{alg2:brunovsky-riccati}. The complexity-wise and procedural comparison between~\cref{alg1:riccati-classic,alg2:brunovsky-riccati} are summarized in~\cref{fig:diagram}. 
\vspace{-0.8em}
\begin{algorithm}[ht]
\begin{spacing}{1}
\caption{Brunovsky Riccati recursion to solve~\Crefrange{eq:lq-ocp-cost}{eq:lq-ocp-dynamics} via the chain~\cref{eq:lq-ocp-all} $\rightarrow$~\cref{eq:lqr-controllable} $\rightarrow$~\cref{eq:lqr-brunovsky}}
\begin{algorithmic}[1]
\Require $x_0, \{Q_k, S_k, R_k, q_k, r_k, b_k\}, A, B$
\State Kalman decomposition as in~\cref{eq:kalman-decomposition-controllable} to get $T_{kd}$. \label{alg2:line:kalman-decompose}
\For{$k = 1:N$ \textbf{in parallel}}
\State Compute $b_k^{c}, b_k^{uc}$ in~\cref{eq:new-offset-controllable}. \label{alg2:line:parallel-offset}
\EndFor
\State Simulate the uncontrollable part in~\cref{eq:uncontrollable-dynamics} to get $\{x^{uc}_k\}$. 
\State Compute $T_{jo}, F_{db}, G$ in~\cref{eq:transformation-brunovsky} to get Brunovsky $(A_b, B_b)$. \label{alg2:line:brunovsky-transformation}
\State Compute initial state $z_0$ using~\cref{eq:new-controllable-cost-initial-state} and then~\cref{eq:new-brunovsky-cost-offset-initial-state}. 
\For{$k = 1:N$ \textbf{in parallel}}
\State Compute $Q_k^c, S_k^c, q_k^c, b_k^{co}$ in~\cref{eq:new-controllable-cost}. \label{alg2:line:parallel-controllable-cost}
\State Compute $\tilde{Q}_k, \tilde{R}_k, \tilde{S}_k, \tilde{q}_k, \tilde{r}_k, \tilde{b}_k$ in~\cref{eq:new-brunovsky-cost}.\label{alg2:line:parallel-brunovsky-cost}
\EndFor
\State Pass $z_0, \{\tilde{Q}_k, \tilde{R}_k, \tilde{S}_k, \tilde{q}_k, \tilde{r}_k, \tilde{b}_k\}, A_b, B_b$ to~\cref{alg1:riccati-classic}. Compute $A\tp P A, B\tp P A, B\tp P B$ with~\cref{eq:sparsity-block,eq:sparsity-pattern}. Get the solution $\{z_k^*, v_k^*\}$. \label{alg2:line:riccati-less-flops}
\For{$k = 1:N$ \textbf{in parallel}}
\State Recover the solution $\{x_k^{c*}, u_k^{c*}\}$ to~\cref{eq:lqr-controllable} with~\cref{eq:solution-recover-brunovsky}. \label{alg2:line:parallel-solution-controllable}
\State Recover the solution $\{x_k^{*}, u_k^{*}\}$ to~\Crefrange{eq:lq-ocp-cost}{eq:lq-ocp-dynamics} with~\cref{eq:solution-recover-controllable}.\label{alg2:line:parallel-solution-original}
\EndFor
\Ensure $\{x_k^*, u_k^*\}$
\end{algorithmic}
\label{alg2:brunovsky-riccati}
\end{spacing}
\end{algorithm}
\vspace{-0.8em}

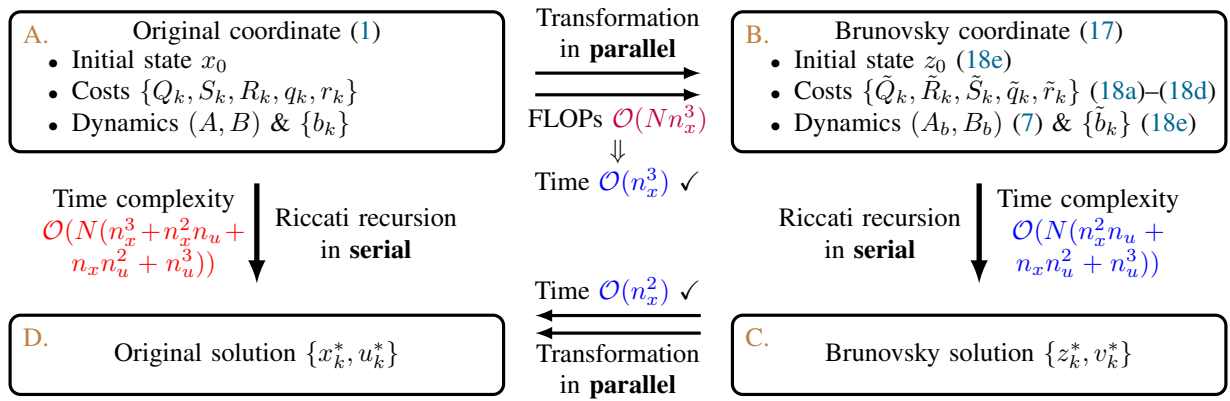
\begin{figure*}[bp]
\vspace{-1em}
\centering
\begin{tikzpicture}[
        node distance=4ex and 0em,
        block/.style={rectangle, draw, fill=white!20, 
    text width=18em, text centered, rounded corners, minimum height=3em},
        line/.style={draw, -latex},
        very thick,
        decoration={
        markings,
        mark=at position 0.5 with {\arrow{latex}}}
        ]

        \node [block] (1) {Original coordinate~\cref{eq:lq-ocp-all} \vbox{ \begin{itemize}
            \item Initial state $x_0$ 
            \item Costs $\{Q_k, S_k, R_k, q_k, r_k\}$
            \item Dynamics $(A, B)$ \& $\{b_k\}$
        \end{itemize}}};
        \node[below right, inner sep=5pt] at (1.north west) {\textcolor{brown}{A.}};
        
        \node [block, right=of 1, xshift=3cm] (2) {Brunovsky coordinate~\cref{eq:lqr-brunovsky} \vbox
            {\begin{itemize}
                \item Initial state $z_0$~\cref{eq:new-brunovsky-cost-offset-initial-state}
                \item Costs $\{\tilde{Q}_k, \tilde{R}_k, \tilde{S}_k, \tilde{q}_k, \tilde{r}_k\}$~\Crefrange{eq:new-brunovsky-cost-Q}{eq:new-brunovsky-cost-q}
                \item Dynamics $(A_b, B_b)$~\cref{eq:brunovsky-form} \& $\{\tilde{b}_k\}$~\cref{eq:new-brunovsky-cost-offset-initial-state}
            \end{itemize}}
        };
        \node[below right, inner sep=5pt] at (2.north west) {\textcolor{brown}{B.}};

        \node [block, below=of 2, yshift=-1.5cm] (3) {Brunovsky solution $\{z_k^*, v_k^*\}$};
        \node[below right, inner sep=5pt] at (3.north west) {\textcolor{brown}{C.}};

        \node [block, left=of 3, xshift=-3cm] (4) {Original solution $\{x_k^*, u_k^*\}$};
        \node[below right, inner sep=5pt] at (4.north west) {\textcolor{brown}{D.}};

        \draw[-latex] ([xshift=.4cm]1.362) -- ([xshift=-.4cm]2.178) 
        node [text width=3cm,midway,above,align=center ] 
        {Transformation in \textbf{parallel}};;
        \draw[-latex] ([xshift=.4cm]1.357) -- ([xshift=-.4cm]2.183) 
        node [text width=3cm,midway,below,align=center ] 
        {FLOPs $\mathcolor{purple}{\mathcal{O}(Nn_x^3)}$ \\
        $\Downarrow$ \\
        Time $\mathcolor{blue}{\mathcal{O}(n_x^3)}$ \checkmark};;
        
        \draw[-latex,line width=2pt] ([yshift=-.4cm]2.270) -- ([yshift=.4cm]3.90) 
        node [text width=2.6cm,midway,right,align=center ]
        {Time complexity $\mathcolor{blue}{\mathcal{O}(N(n_x^2 n_u + n_x n_u^2 + n_u^3))}$}
        node [text width=2.6cm,midway,left,align=center ] 
        {Riccati recursion \\
        in \textbf{serial}};;
        
        \draw[-latex] ([xshift=-.4cm]3.171) -- ([xshift=.4cm]4.009)
        node [text width=3cm,midway,above,align=center ] 
        {
        Time $\mathcolor{blue}{\mathcal{O}(n_x^2)}$ \checkmark \\
        };;   
        \draw[-latex] ([xshift=-.4cm]3.175) -- ([xshift=.4cm]4.005)   node [text width=3cm,midway,below,align=center ] 
        {Transformation in \textbf{parallel}};;  
              
        \draw[-latex,line width=2pt] ([yshift=-.4cm]1.270) -- ([yshift=.4cm]4.90) 
        node [text width=2.7cm,midway,left,align=center ] 
        {Time complexity $\mathcolor{red}{\mathcal{O}(N(n_x^3 + n_x^2 n_u + n_x n_u^2 + n_u^3))}$}
        node [text width=2.6cm,midway,right,align=center ] 
        {Riccati recursion \\
        in \textbf{serial}};;
\end{tikzpicture}
\caption{Classical and Brunovsky Riccati recursion, omitting the intermediate~\cref{eq:lqr-controllable}. Complexity of~\cref{alg1:riccati-classic} (\textcolor{brown}{A} $\rightarrow$ \textcolor{brown}{D}): $\mathcolor{red}{\mathcal{O}(N(n_x^3 + n_x^2 n_u + n_x n_u^2 + n_u^3))}$. Complexity of~\cref{alg2:brunovsky-riccati} (\textcolor{brown}{A} $\rightarrow$ \textcolor{brown}{B} $\rightarrow$ \textcolor{brown}{C} $\rightarrow$ \textcolor{brown}{D}): $\mathcolor{blue}{\mathcal{O}(n_x^3 + N(n_x^2 n_u + n_x n_u^2 + n_u^3))}$}
\label{fig:diagram}
\end{figure*}


\cref{alg2:line:riccati-less-flops} dominates the time complexity of~\cref{alg2:brunovsky-riccati}, which is still a Riccati recursion but accelerated by the sparsity of Brunovsky form (block-diagonality and zero-one pattern). The per iteration complexity in~\cref{eq:flop-count-riccati} is dramatically reduced to $\alpha = 0, \beta = 2, \gamma = 1, \delta = \nicefrac{1}{3}$, bringing the total time complexity for horizon $N$ down to $N (2n_x^2 n_u + n_x n_u^2 + \frac{1}{3}n_u^3)$. $n_x^3$ is hidden by the $n_x^2$ memory copy and this can be a significant improvement because $n_x \gg n_u$ often prevails. 

In~\cref{alg2:line:kalman-decompose,alg2:line:brunovsky-transformation} of~\cref{alg2:brunovsky-riccati}, the transformations $T_{kd}, T_{jo}, F_{db}, G$ are computed \textbf{only once}. 
The time complexity is $\mathcal{O}(n_x^3 + n_x^2 n_u + n_x n_u^2 + n_u^3)$. Different methods change the coefficients, but in general they are insignificant compared with the Riccati recursion. 

The key success of~\cref{alg2:brunovsky-riccati} is to distribute the original serial $\mathcal{O}(N n_x^3)$ complexity (\cref{alg1:line:APA} of~\cref{alg1:riccati-classic}) to $N$ threads running in \textbf{parallel} and \textbf{a priori} (\cref{alg2:line:parallel-controllable-cost,alg2:line:parallel-brunovsky-cost} of~\cref{alg2:brunovsky-riccati}). \cref{fig:diagram} sheds light on the intuition. The time invariancy of the linear system guarantees the transformations $T_{kd}, T_{jo}, F_{db}, G$ are constant along the horizon thus data parallelism can be implemented ahead of the sequentially executed Riccati recursion. Though more floating point operations are conducted in total due to extra transformations, the overall time complexity reduces to $\mathcal{O}(n_x^3)$ without $N$ involved, if there are $N$ threads/cores available for parallelism. 

If only $T_{jo}, F_{db}$ are used in~\cref{eq:transformation-brunovsky} (as in the newly proposed perspective), the cost weights in~\cref{eq:lqr-brunovsky} remain unchanged. Computations in~\cref{eq:new-brunovsky-cost-R-r} are avoided. However, it also leads to larger $\beta, \gamma$ in~\cref{eq:flop-count-riccati} (because $T_{jo} B_{co} \neq B_b$ is dense so~\cref{eq:sparsity-block,eq:sparsity-pattern} apply only to $A_b$), which is NOT preferred because the serial Riccati recursion takes the most amount of time. 

$T_{kd}, T_{jo}, F_{db}$ are dense. This implies that the potential diagonality of $Q_k$ in~\cref{eq:lq-ocp-cost} will be destroyed by~\cref{eq:new-controllable-cost-Q} and~\cref{eq:new-brunovsky-cost-Q}. However, this has no impact on the complexity up to cubic order because matrix addition in~\cref{alg1:line:APA} of~\cref{alg1:riccati-classic} costs $n_x^2$ only. Similar consequence applies to $G$ and $R_k$. 
\subsection{Inequality constrained LQ OCP}

It is well-known that the linear inequality in~\cref{eq:lq-ocp-inequality,eq:ocp-qp-constraints} will not destroy the stage-wise structure in all kinds of optimization algorithms (active set, interior point, ADMM, etc. ). The slacks and Lagrangian multipliers affect the diagonal blocks of $H$ and the vector $g$ in~\cref{eq:kkt-system}. The numerical values of $Q_k, R_k, S_k, q_k, r_k$ in~\cref{eq:lq-ocp-cost} are modified but not the OCP QP structure. As a result, the chain of transformation~\cref{eq:lq-ocp-all} $\rightarrow$~\cref{eq:lqr-controllable} $\rightarrow$~\cref{eq:lqr-brunovsky} is applicable to linear inequalities~\cref{eq:lq-ocp-inequality}. 

Denote $C T_{kd}\iv = \begin{bmatrix}
    C^c & C^{uc}
\end{bmatrix}$ with $C^c \in \mathbb{R}^{n_i \times n_x^c}, C^{uc} \in \mathbb{R}^{n_i \times n_x^{uc}}$, the inequality that will be part of~\cref{eq:lqr-controllable} becomes: 
\begin{equation}
    C^c x_k^{c} + D u_k \leq d - C^{uc} x_k^{uc}
    \tag{11c}
    \label{eq:inequality-controllable-coordinate}
\end{equation}
Similarly, in~\cref{eq:lqr-brunovsky}, the inequality becomes
\begin{equation}
    (C^c + D F_{db})T_{jo}\iv z_k + D G v_k \leq d - C^{uc} x_k^{uc}
    \tag{17c}
    \label{eq:inequality-brunovsky-coordinate}
\end{equation}

For general non-zero $C, D$ in~\cref{eq:lq-ocp-inequality}, not much is changed compared with~\cref{eq:inequality-brunovsky-coordinate}. However, if $C=0$, i.e., only box constraints are presented for inputs, $F_{db}$ will wipe the possibility of a projection method in ADMM~\cite{o2013splitting}. Nevertheless, it has little influence in an interior point method setup. 

\section{NUMERICAL RESULTS}
\begin{figure}[ht]
    \centering
    \includegraphics[width=1.0\linewidth]{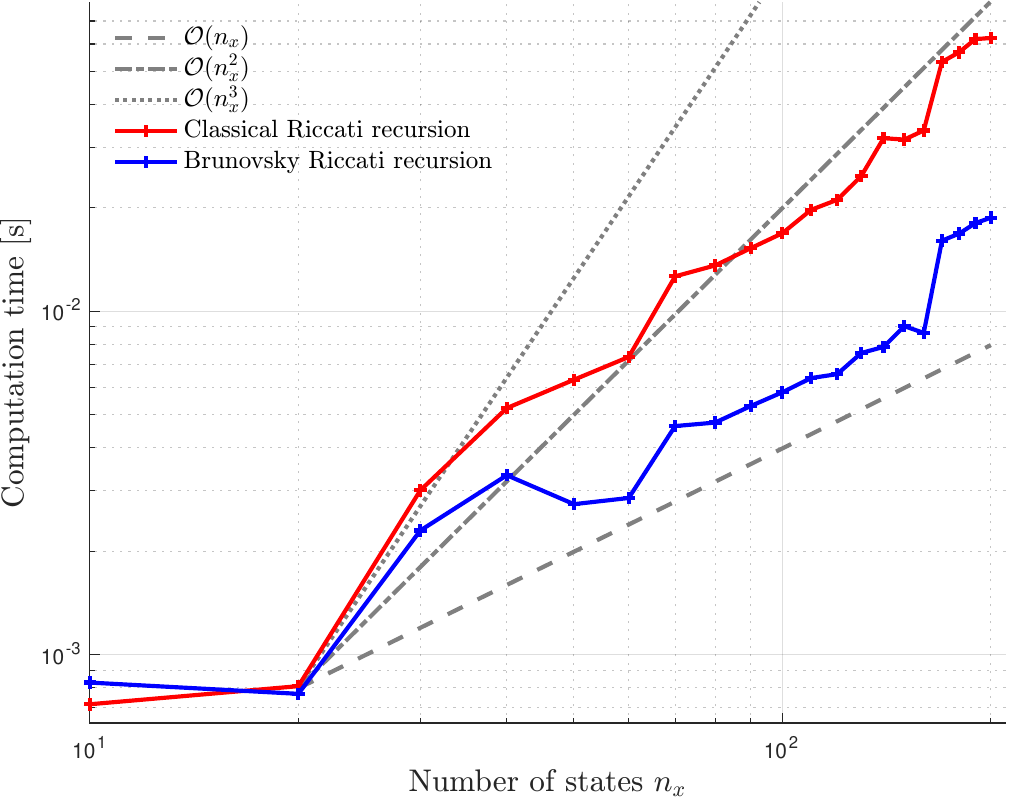}
    \caption{Classical/Brunovsky Riccati recursion comparison. Red curve:~\cref{alg1:riccati-classic}. Blue curve:~\cref{alg2:brunovsky-riccati}. Grey curves: auxiliary lines. }
    \label{fig:riccati-comparison}
\end{figure}
\vspace{-1em}
A numerical experiment was conducted using Matlab on MacBook Pro with Apple M1 Max chip (10-core CPU). Random controllable pairs $(A_{co}, B_{co})$ and sets of cost coefficients $\{Q_k^c, R_k^c\}$ are generated for fixed horizon length $N = 50$, control input $n_u = 10$, and varying state size $n_x \in \{10k, 20 \geq k \geq 1\}$. \cref{alg1:riccati-classic,alg2:brunovsky-riccati} are run for 100 times for different setups and the average running time is calculated. The results are summarized in~\cref{fig:riccati-comparison}. 

The blue curve is lower than the red one when $n_x > 20$, illustrating the time efficiency of our algorithm. When $n_x$ is relatively small, the red curve matches $\mathcal{O}(n_x^3)$ and the blue curve matches $\mathcal{O}(n_x^2)$. As $n_x$ grows larger, the internal matrix acceleration tricks of Matlab dominate and bend the respective curves to $\mathcal{O}(n_x^2)$ and $\mathcal{O}(n_x)$. 

The current implementation leverages the sparse matrix class rooted in Matlab, so it only takes advantage of the block-diagonality. If dedicated linear algebra routines are implemented in C\texttt{++}, the zero-one pattern will further accelerate the computation. This, together with inequality constraints in the interior point method framework and conditioning analysis for $T_{jo}, F_{db}, G$, are left for future work. 

\section{CONCLUSIONS}

In this paper, we propose a novel \textit{Brunovsky Riccati recursion} algorithm for linear quadratic optimal control problems with time invariant systems, which is significantly faster than the state-of-the-art Riccati solver. This is achieved with transformation of LTI systems to a controllable form then to the Brunovsky form, sparsity exploitation of the Brunovsky form with a custom linear algebra routine, and parallel computation before and after the Riccati recursion. We also propose a new insight to transform arbitrary controllable linear systems to Brunovsky form, where deadbeat control and Jordan normal form bridge the gap. 



\addtolength{\textheight}{-12cm}   

\bibliographystyle{IEEEtran}
\bibliography{IEEEabrv,citation.bib}

\end{document}